\documentclass{amsart}

\usepackage{amscd,amssymb,amsopn,amsmath,amsthm,graphics,amsfonts,enumerate,verbatim,calc}
\usepackage[dvips]{graphicx}
\usepackage{mathrsfs}
\usepackage[all]{xypic}

\usepackage{amssymb,amsmath, color}
\newtheorem{theorem}{Theorem}[section]
\newtheorem{lemma}[theorem]{Lemma}

\newtheorem{proposition}[theorem]{Proposition}

\theoremstyle{definition}

\theoremstyle{remark}
\newtheorem{remark}[theorem]{Remark}
\newtheorem{example}[theorem]{Example}

\newcommand{\Ass}{\operatorname{Ass}}
\newcommand{\Dim}{\operatorname{dim}}
\newcommand{\injdim}{\operatorname{inj.dim}}

\newcommand{\I}{\operatorname{I}}

\newcommand{\Spec}{\operatorname{Spec}}

\newcommand{\Ht}{\operatorname{ht}}

\newcommand{\Ext}{\operatorname{Ext}}
\newcommand{\Supp}{\operatorname{Supp}}

\newcommand{\Hom}{\operatorname{Hom}}

\newcommand{\depth}{\operatorname{depth}}

\newcommand{\fm}{\frak{m}}
\newcommand{\fp}{\frak{p}}
\newcommand{\fq}{\frak{q}}

\begin{document}

\author[M. Dorreh]{Mehdi Dorreh}

\address{ Iran National Science Foundation: INSF and Department of Mathematics, University Of Isfahan, P. O. Box: 81746-73444, Isfahan, Iran. 
}
\email{mdorreh@ipm.ir}

\title{On the injective dimension of $\mathscr{F}$-finite modules and holonomic $\mathscr{D}$-modules}

\subjclass[2010]{Primary 13N10. Secondary 13D45.}

\begin{abstract}
Let $R$ be a regular local ring containing a field $k$ of characteristic $p$ and $M$ be an $\mathscr{F}$-finite module. In this paper we study the injective dimension of $M$. We prove that $\Dim_R(M) -1 \leq \injdim_R(M)$.  If $R = k[[x_1,\ldots,x_n]]$ where $k$ is a field of characteristic $0$ we prove the analogous result for a class of holonomic $\mathscr{D}$-modules  which contains local cohomology modules.   
\end{abstract}

\maketitle

\section{Introduction}

Throughout this paper, $R$ is a commutative Noetherian ring with unit. If $M$ is an $R$-module and $\I \subset R$ is an ideal, we denote  the $i$-th local cohomology  of $M$ with support in $\I$ by $H_{\I}^i(M)$. 
 
In a remarkable paper, \cite{L},  Lyubeznik    used $\mathscr{D}$-modules to prove if $R$ is any regular ring containing a field of characteristic $0$ and  $\I$ is an ideal of $R$, then 
\begin{itemize}
\item[a)]$ H_{\fm}^{i}(H_{\I}^{i}(R))$ is injective for every maximal ideal $\fm$  of $R$. 
\item[b)] $\injdim_{R}(H_{\I}^{i}(R)) \leq \Dim_{R}(H_{\I}^i(R))$. 
\item[c)] For every maximal ideal $\fm$ of $R$ the set of associated primes of $H_{\I}^{i}(R)$ contained in $\fm$ is finite.
\item[d)] All the bass numbers of $H_{\I}^{i}(R)$ are finite.
\end{itemize} 
Here $\injdim_{R}(H_{\I}^{i}(R))$ stands for the injective dimension of $H_{\I}^i(R)$,
$\Dim_{R}(H_{\I}^i(R))$ denotes the dimension of the support of $H_{\I}^i(R)$ in $\Spec(R)$ and the $j$-th Bass number of an $R$-module $M$ with respect to a prime ideal $\fp$ is defined as $\mu^j(\fp,M) = \Dim_{k(\fp)} \Ext_{R_\fp}^j(k(\fp),M_\fp)$ where $k(\fp)$ is the residue field of $R_\fp$.

The analogous results had proved earlier for regular local ring of positive characteristic by Huneke and Sharp \cite{HS}, with using of Frobenius functor.

Later Lyubeznik \cite{L1} developed the theory of $\mathscr{F}$-modules over regular ring of char $ p >0$ and proved the same results in char $p > 0$. The theory of $\mathscr{F}$-modules turned out to be very effective. For example Lyubeznik and etc \cite{BBLSZ} used $\mathscr{D}$-modules over $\mathbb{Z}$ and $\mathbb{Q}$ along with the theory of $\mathscr{F}$-modules to prove if $R$ is a smooth $\mathbb{Z}$-algebra and $I$ an ideal of $R$  then the set of associated primes of local cohomology module $H_{I}^{i}(R)$ is finite. 

By Lyubeznik results, the injective dimension of $H_{\I}^i(R)$  is bounded by its dimension. More generally if $M$ is an $\mathscr{F}$-module over a regular ring of positive characteristic or is a $\mathscr{D}$- module over power series ring $k[[x_1,\ldots,x_n]]$ where $k$ is a field of char $0$, then the injective dimension of $M$  is bounded by its dimension, see \cite[Theorem 1.4]{L1} and \cite[Theorem 2.4 (b)]{L} . A question of  Hellus \cite{hel} asks when $\injdim_{R}(H_{\I}^{i}(R)) =\Dim_{R}(H_{\I}^i(R))$.
 He proved the equality $\injdim_{R}(H_{\I}^{i}(R)) =\Dim_{R}(H_{\I}^i(R))$ for a regular local ring $R$ which contains a field and cofinite local cohomology $H_{I}^i(R)$, see \cite[Corollary 2.6]{hel}. On the other hand, he presented  counterexamples for this equality in which $\injdim_{R}(H_{\I}^{i}(R)) =0 $ but $\Dim_{R}(H_{\I}^i(R)) =1$, see \cite[Example 2.9, 2.11]{hel}. Also for polynomial ring $ R = k[x_1,\ldots,x_n]$ with  field $k$ of characteristic zero, Puthenpurakal, \cite[Corollary 1.2]{puthen}, proved  $\injdim_{R}(H_{\I}^{i}(R)) = \Dim_{R}(H_{\I}^i(R))$.  
 
 In this paper, motivated by these results, we attempt to obtain  lower bound for the injective dimension of $\mathscr{F}$-modules and $\mathscr{D}$-modules. We succeed in this for a subclass of $\mathscr{F}$-modules called $\mathscr{F}$-finite and subclass of $\mathscr{D}$-modules which contains local cohomology modules. In fact we prove that

 \begin{theorem}(Theorem \ref{14})\label{0}
 Let $(R,\fm)$ be a regular local ring which contains a field. Let $\I$ be an ideal of $R$.  The following hold.
 \begin{itemize}
 \item[(i)] Assume characteristic of $R$ is $p >0$ and $M$ is an $\mathscr{F}$-finite module.
 Then $\;\;\;\;\;\dim_R M-1 \leq\injdim_R M$.
 \item[(ii)]Assume characteristic of $R$ is $0$ and $M = H_{\I}^i(R)_{f} $ for some $f \in R$. Then $\;\;\;\;\;\dim_R M-1 \leq\injdim_R M$.
 \end{itemize}
 \end{theorem}
  
 This manuscript is organized as follows. In section 2, we recall some definitions and properties of $\mathscr{D}$-modules and $\mathscr{F}$-modules. Later, in section 3, we discuss some lemmas and propositions which will help us in proving our main theorem. In section 4, we prove our main theorem.

 \section{Preliminaries}
 
Throughout this paper, we always assume that $R$ is a regular local ring which contains a field. In this section we review the theory of $\mathscr{D}$-modules and $\mathscr{F}$-modules and state two usefule lemmas.   
 
\textbf{$\mathscr{D}$-modules.}
Let $k$ be a field of characteristic $0$ and let $R$ denote the formal power series ring $k[[x_1,\ldots,x_n]]$ in $n$ variables over $k$. Let $\mathscr{D} = \mathscr{D}(R,k)$ denote the subring of the $k$-vector space endomorphisms of $R$ generated by $R$ and 
 the usual differential operators $\delta_1,\ldots,\delta_n$, defined formally, so that $\delta_if = \frac{\partial f}{\partial x_i}$. We simply say $\mathscr{D}$-modules for left $\mathscr{D}(R,k)$ modules. $\mathscr{D}(R,k)$ is left and right Noetherian \cite[Lemma 3.1.6]{bjork}. This implies that every finitely generated $\mathscr{D}$-module is Noetherian. 
 The natural action of $\mathscr{D}(R,k)$ on $R$ makes $R$ as a $\mathscr{D}$-module. In addition if $M$ is a $\mathscr{D}$-module and $S \subset R$ is a multiplicative system of elements, using the quotient rule,   $M_S$ carries a natural structure of $\mathscr{D}$-module. Let $\I$ be an ideal of $R$. The \v{C}ech complex on a generating set for $\I$ is a complex of $\mathscr{D}$-modules; it then follows that each local cohomology module $H_{\I}^i(R)$ is a $\mathscr{D}$-module.
 
 We will use the following several times in this paper. 
\begin{remark}\label{1}
\begin{itemize}
 Adopt the above notations.
 \item[a)] Let $M$ be a $\mathscr{D}$-module. Then $\injdim_R M \leq \dim_R M$ \cite[Theorem 2.4(b)]{L}.
 \item[b)] Let $M$ be a $\mathscr{D}$-module and $I$ be an ideal of $R$. Then $H_{I}^i(M)$ have a natural structure of $\mathscr{D}$-modules \cite[Example 2.1(iv)]{L}. In particular $\Gamma_I(M)$ is a $\mathscr{D}$-submodule of $M$ where $\Gamma_I$ is the $I$-torsion functor .
\item[c)] Let $\fp$ be a prime ideal of $R$ and let   $E_R(R/\fp)$ denote  the injective envelope of $R/\fp$. Assume $\Ht_R (\fp) = d$. Recall that $E_R(R/\fp) = H_{\fp}^d(R)_\fp$. It follows that $E_R(R/\fp)$ is a $\mathscr{D}$-module and the natural inclusion $H_{\fp}^d(R) \to E_R(R/\fp)$ is $\mathscr{D}(R,k)$-linear. 
\item[d)] Let $(S,\fm)$ be a regular local ring which contains a field of characteristic zero. We denote by $\hat{S}$ the completion of $S$ with respect to the  maximal ideal $\fm$. By Cohen structure theorem $\hat{S} = k[[x_1,\ldots,x_n]]$ where $k$ is a field of characteristic zero. Let $\fp$ be the prime ideal of $S$ such that $\Ht_S(\fp) = d$. Recall that $E_S(S/\fp) = H_{\fp}^d(S)_{\fp}$. Then $ E_S(S/\fp) \otimes_{S} \hat{S} \cong H_{\fp \hat{S}}^d(\hat{S})_{\fp}$, see \cite[Theorem 4.3.2]{BS}. Hence $E_S(S/\fp) \otimes_{S} \hat{S}$ has a natural structure of $\mathscr{D}(\hat{S},k)$-module.

\end{itemize}   
 \end{remark} There exists a remarkable class of finitely generated $\mathscr{D}$-modules, called holonomic $\mathscr{D}$-modules. See  \cite[Definition 7.12]{bjork} for a definition of a holonomic $\mathscr{D}$-module. 
 \begin{remark}\label{2}
 Some of the properties of holonomic modules are as follows:
 \begin{itemize}
\item[a)] $R$ with its natural structure of $\mathscr{D}(R,k)$-module is holonomic\cite[Theorem 3.3.2]{bjork}. 
\item[b)] If $M$ is holonomic and $f\in R$, then $M_f$ is holonomic \cite[Theorem 3.4.1]{bjork}.
\item[c)] Let $M$ be a holonomic $\mathscr{D}$-module. Assume $\Ass_R M= \{\fp\}$ and $M$ is $\fp$-torsion. Then there exists $h \in  R \setminus \fp$ such that $\Hom_R(R/\fp,M)_h$ is finitely generated as an $R_h$-module \cite[Proposition 2.3]{puthen}.
\item[d)]The holonomic modules form an abelian subcategory of the category  of $\mathscr{D}$-modules, which is closed under formation of submodules, quotient modules and extensions \cite[2.2 c]{L}. So $H_{\I}^i(R)$ is a  holonomic $\mathscr{D}$-module. 
\item[e)] If $M$ is holonomic, then $H_I^i(M)$ is holonomic \cite[2.2 d]{L}.
\item[f)] If $M$ is holonomic, all the Bass numbers of $M$ are finite \cite[Theorem 2.4(d)]{L}.
\item[g)] If $M$ is holonomic, the set of the associated primes of $M$ is finite \cite[Theorem 2.4(c)]{L}.
\end{itemize}
\end{remark}
\textbf{$\mathscr{F}$-modules}. The notion of $\mathscr{F}$-modules was introduced by Lyubeznik in \cite{L1}. We collect some notations and preliminary results from \cite{L1}. Let $R$ be a regular ring containing a field of characteristic $p >0$. Let $R^ \prime$ be the additive group of $R$ regarded as an $R$-bimodule with the usual left $R$-action and with the right $R$-action defined by $r^\prime r = r^p r^\prime$ for all $r \in R$, $ r^\prime \in R^\prime$. For an $R$-module $M$, define $F(M) = R^{\prime} \otimes_R M$ ; we view this as an $R$-module via the left $R$-module structure on $R^\prime$.

An $\mathscr{F}_R$-module $M$ is an $R$-module $M$ with an $R$-module isomorphism $\theta : M \to F(M)$ which is called the structure morphism of $M$. We will abbreviate $\mathscr{F}_R$ to $\mathscr{F}$ for the sake of readability (if this causes no confusion). A homomorphism of $\mathscr{F}$-modules is an $R$-module homomorphism $f : M \to M^\prime$ such that the following diagram commutes (where $\theta$ and $\theta^\prime$ are the structure morphisms of $M$ and $M^\prime$).   
\[\xymatrix{
M \ar[r]^{f} \ar[d]^{\theta} & M^\prime  \ar[d]^{\theta^\prime} \\
F(M) \ar[r]^{F(f)} & F(M^\prime)
}\] 
It is not hard to see that the category of $\mathscr{F}$-modules is abelian. 
\begin{remark}\label{3}
 Some of the properties of $\mathscr{F}$-modules are as follows:
 \begin{itemize}
\item[a)] The ring $R$ has a natural $\mathscr{F}$-module structure \cite[Example 1.2(a)]{L1}.
\item[b)] Let $I$ be an ideal of $R$ and $M$ be an $\mathscr{F}$-module. Then an $\mathscr{F}$-module structure on an $R$-module $M$ induces an $\mathscr{F}$-module structure on the local cohomology module $H_{I}^i(M)$. In particular $\Gamma_I(M)$ is an $\mathscr{F}$-submodule of $M$ \cite[Example 1.2(b)]{L1}.
\item[c)] If $M$ is an $\mathscr{F}$-module and $0 \to M \to E^{\bullet}$ is the minimal injective resolution of $M$ in the category of $R$-modules, then each $E^i$ acquires a structure of $\mathscr{F}$-module such that the resolution becomes a complex of $\mathscr{F}$-modules and $\mathscr{F}$-module homomorphisms \cite[Example 1.2($b^{\prime\prime}$)]{L1}. 
\item[d)] Let $M$ be an $\mathscr{F}$-module. Then $\injdim_R M \leq \Dim_R  M$ \cite[Theorem 1.4]{L1}. 
\item[e)] Let $M$ be an $\mathscr{F}$-module and $S \subset R$ be a multiplicative set. Then $M_S$ has a natural structure of $\mathscr{F}$-module such that the natural localization map $M \to M_S$ is the $\mathscr{F}$-module homomorphism \cite[Proposition 1.3 (b)]{L1}.  
\end{itemize}
\end{remark}
There exists an important  class of $\mathscr{F}$-modules, called $\mathscr{F}$-finite modules. See \cite[Definition 2.1]{L1} for a definition of an $\mathscr{F}$-finite module.  
\begin{remark}\label{4}
 Some of the properties of $\mathscr{F}$-finite modules are as follows:
 \begin{itemize}
\item[a)] The $\mathscr{F}$-finite modules form a full abelian subcategory of the category of $\mathscr{F}$-modules which is closed under formation of submodules, quotient modules and extensions \cite[Theorem 2.8]{L1}.
\item[b)] If $M$ is  an $\mathscr{F}$-finite module, then $M_f$ is $\mathscr{F}$-finite, where $f \in R$ \cite[Proposition 2.9(b)]{L1}.
\item[c)] If $M$ is an $\mathscr{F}$-finite module and $I  $ is an ideal of $R$, then $ H_I^i(M)$ with its induced $\mathscr{F}$-module structure is $\mathscr{F}$-finite \cite[Proposition 2.10]{L1}. 
\item[d)] All the Bass numbers of an $\mathscr{F}$-finite module  $M$ are finite  \cite[Theorem 2.11]{L1}. 
\item[e)] The set of the associated primes of an  $\mathscr{F}$-finite module  $M$ is finite  \cite[Theorem 2.12]{L1}. 
\item[f)] If $M$ is  an $\mathscr{F}_{R}$-finite module, then $M_{\fp}$ is $\mathscr{F}_{R_\fp}$-finite, where $\fp \in \Spec(R)$ \cite[Proposition 2.9(a)]{L1}.
\end{itemize}
\end{remark}
For the convenience of the reader, we state the following proved facts.
\begin{lemma}\label{15}
Let $R$ be a Noetherian local ring which has a finitely generated injective module. Then $R$ is an Artinian ring.
\end{lemma}
\begin{proof}
By \cite[Theorem 3.1.17]{BH}, $\depth R =0$. Also well-known proved conjecture of Bass implies that $R$ is Cohen-Macaulay. Then $\dim R =0$.
\end{proof}
\begin{lemma}\label{16}
Let $R \to S $ be a faithfully flat map of  Noetherian  rings. Then an   $R$ module $L$ is finitely generated if and only if $L\otimes_R S$ is finitely genrated as an $S$-module. 
\end{lemma}
\begin{proof}
See \cite[Proposition 3.3]{puthen}.
\end{proof}

\section{Preliminary  lemmas}
In this section, our objective is to prove Proposition \ref{12}  which will enable us to prove the main theorem in the next section. Let $(R,\fm)$ be a local ring and $M$ be an $R$-module. By $\depth_R(M)$, we mean the length of the maximal $M$-regular sequence in $\fm$.
\begin{lemma}\label{5}
Let $k$ be a field of characteristic zero and $R = k[[x_1,\ldots,x_n]]$. Let $\fp$ be a prime ideal of $R$ of height less than $n-1$. Then $E_R(R/\fp)$ is not a holonomic $\mathscr{D}$-module.
\end{lemma}
\begin{proof}
Suppose on the contrary $E_R(R/\fp)$ is a holonomic $\mathscr{D}$-module. It is well-known that $\Gamma_\fp(E_R(R/\fp)) = E_R(R/\fp)$ and $\Ass_R E(R/\fp) = \fp$. Then by  Remark \ref{2}(c), there exists $h \in R\setminus \fp$ such that $\Hom_R(R/\fp , E_R(R/\fp))_h$ is a finitely generated $R_h$-module. Pick $\fq \in \Spec (R)$ which contains $\fp$ such that $\Ht_R(\fq) = n-1$ and $h \notin \fq$. It follows that $ M := \Hom_R(R/\fp , E_R(R/\fp))_\fq$ is a non-zero finitely generated  $R_\fq$-module. On the other hand $M$ is an injective $R_\fq/\fp R_\fq$-module . Then, in view of Lemma \ref{15}, $R_\fq/\fp R_\fq$ is an Artininan ring.  This  contradicts with the fact that  $R_\fq /\fp R_\fq$ is a domain of dimension greater than one. 
\end{proof}
Let $\I$ be an ideal of a ring $R$. By $\min_R(\I)$, we mean the set of all minimal prime ideals of $\I$.
\begin{lemma}\label{6}
Let $(R,\fm)$ be a regular local ring of dimension $n$ which contains a field of characteristic zero. Assume $P \in \Spec(R)$ such that $\Ht_R(P) = d  \leq n-2$. Let $\hat{R}$ denote the completion of $R$ with respect to the maximal ideal $\fm$. In view of Remark \ref{1}(d) $E_{R}(R/P) \otimes _{R}{\hat{R}} $ has a natural structure of $\mathscr{D}(\hat{R},k)$-module where $k$ is a suitable coefficient field of $\hat{R}$. Then $E_{R}(R/P) \otimes _{R}{\hat{R}} $ is a non-holonomic $\mathscr{D}$-module.
\end{lemma}
\begin{proof}
Recall that $E_{R}(R/P) \cong H_{P}^{d}(R)_{P}$ and $E_{R}(R/P) \otimes _{R}{\hat{R}} \cong H_{P}^{d}(R)_{P} \otimes_{R} \hat{R} \cong H_{P\hat{R}}^{d}(\hat{R})_{P}$. In view of Remark \ref{1}(d), $E_{R}(R/P) \otimes _{R}{\hat{R}} $ has a natural structure of $\mathscr{D}(\hat{R},k)$-module where $k$ is a field  of characteristic zero which is  contained in $\hat{R}$. We simply say $E_{R}(R/P) \otimes _{R}{\hat{R}} $ is a $\mathscr{D}$-module.  It is obvious that   $\Ht_{\hat{R}}(P\hat{R}) = d$. 
Let $\min_{\hat{R}}(P\hat{R}) = \{ \fq_1,\ldots,\fq_s\}$.   There are infinitely many primes $\fp \in \Spec(R)$  such that $\Ht_R(\fp) = d+1$ and $P \subsetneqq \fp$, see \cite[Theorem 31.2]{Mat}. For  such $\fp$, $ \Ht_{\hat{R}}(\fp \hat{R}) = d+1$ and $ \fp \hat{R} \cap R = \fp$. Thus without loss of generality we can assume that $\Ht_{\hat{R}}(\fq_1) = d$  and there are infinitely many primes $\fq \in \Spec(\hat{R})$  of height $d+1$  which contains $\fq_1$ and $\Ht_R(\fq \cap R) = d+1$ .

Suppose on the contrary that $ H_{P\hat{R}}^{d}(\hat{R})_{P}$ is holonomic.

\textbf{ Claim 1.}   $ H_{\fq_1}^{d}(\hat{R})_{P}$ is holonomic.

The composition of functors $\Gamma_{\fq_1}(-) = \Gamma_{\fq_1}(\Gamma_{P\hat{R}}(-))$ leads to the spectral sequence $E_2^{p,q} = H_{\fq_1}^p(H_{P\hat{R}}^q(\hat{R})) \Rightarrow H_{\fq_1}^{p+q}(\hat{R})$. It follows that $\Gamma_{\fq_1}(H_{P\hat{R}}^d(\hat{R})) = H_{\fq_1}^d(\hat{R})$.  Hence $H_{\fq_1}^d(\hat{R})$ is the $\mathscr{D}$-submodule of $ H_{P\hat{R}}^{d}(\hat{R})$. Therefore $ H_{\fq_1}^{d}(\hat{R})_{P}$ is a holonomic $\mathscr{D}$-module, see Remark \ref{2}(d). This yields the claim. 

\textbf{ Claim 2.} $\Ass_{\hat{R}}( H_{\fq_1}^{d}(\hat{R})_{P}) = \fq_1$.

 Indeed let $m/s \in  H_{\fq_1}^{d}(\hat{R})_{P} $ such that $m \in  H_{\fq_1}^{d}(\hat{R})$ and $s \in R\setminus P$. If $r \in \hat{R}$ such that $r . m/s =0 $, then 
 there exists  $r^{\prime} \in  R\setminus P \subseteq \hat{R} \setminus \fq_1$ such 
 that $r^{\prime} r m =0$. Keep in mind that $\Ass_{\hat{R}} (H_{\fq_1}^{d}(\hat{R})) = \fq_1$. So $r^{\prime} r \in \fq_1$ and thus $r \in \fq_1$. This yields the claim.
  
 Also $\Gamma_{\fq_1}(H_{\fq_1}^{d}(\hat{R})_P) = H_{\fq_1}^{d}(\hat{R})_P$. Then by Remark \ref{2}(c), there exists $ h \in \hat{R}\setminus \fq_1$ such that $\Hom_{\hat{R}}(\frac{\hat{R}} { \fq_1 \hat{R}}, H_{\fq_1}^{d}(\hat{R})_P)_h$ is  a finitely generated $\hat{R}_h$-module. Since $\fq_i \nsubseteq \fq_1$ for all $ 2 \leq i \leq s$, we can pick $t_i \in \fq_i \setminus \fq_1$ for all $ 2 \leq i \leq s$. Thus $t = t_2 \ldots t_s h \notin \fq_1$. Note that the set of minimal prime ideals of the ideal generated by $t$ and $\fq_1$ is finite. Then  by assumption on choosing $\fq_1$, we  can pick $\fq  \in \Spec(\hat{R})$  of height $d+1$ which contains $\fq_1$ and $t \notin \fq$ such that $\Ht_R(\fq \cap R)  = d+1$.  

Thus $\Hom_{\hat{R}_\fq}(\frac{\hat{R}_\fq} { \fq_1 \hat{R}_\fq},( H_{\fq_1}^{d}(\hat{R})_P)_\fq)$ is a finitely generated $\hat{R}_\fq$-module. 
Since $\min_{\hat{R}_\fq}(P \hat{R}_\fq) = \fq_1 \hat{R}_\fq$, then $  H_{\fq_1  \hat{R}_\fq}^{d}(\hat{R}_\fq) =  H_{P \hat{R}_\fq}^{d}(\hat{R}_{\fq})$.
 Also $ \frac{\hat{R}_\fq} { P\hat{R}_\fq}$ has a filtration of $\hat{R}_\fq$-modules such that quotients of it are  isomorph to $\frac{\hat{R}_\fq} { \fq_1 \hat{R}_\fq}$ or $\frac{\hat{R}_\fq} { \fq \hat{R}_\fq}$, as $\hat{R}_\fq$-module. Thus $\Hom_{\hat{R}_\fq}(\frac{\hat{R}_\fq} { P \hat{R}_\fq},( H_{P\hat{R}}^{d}\hat{R}_P)_\fq)$ is a finitely generated $\hat{R}_\fq$-module.
 
 Look at the faithfully flat  map $R_{\fq \cap R} \to \hat{R}_\fq$. We have following isomorphisms:
 \[\begin{array}{ll} \Hom_{R_{\fq \cap R}}(\frac{R_{\fq \cap R}}{PR_{\fq \cap R}} , (H_{P}^{d}(R)_P)_{\fq 
 \cap R}) \otimes_{R_{\fq \cap R}} \hat{R}_\fq \cong 
  \Hom_{\hat{R}_\fq}(\frac{R_{\fq 
 \cap R}}{PR_{\fq \cap R}}  \otimes_{R_{\fq \cap R}} \hat{R}_{\fq}, (H_{P}^{d}
 (R)_P)_{\fq \cap R} \otimes_{R_{\fq \cap R}} \hat{R}_\fq)\\
  \cong \Hom_{\hat{R}_\fq}
 ((R/P \otimes_{R} R_{\fq \cap R})\otimes_{R_{\fq \cap R}} \hat{R}_{\fq},((H_{P}
 ^{d}(R) \otimes_{R} R_P) \otimes_{R} R_{\fq \cap R}) \otimes_{R_{\fq \cap R}}  
 \hat{R}_\fq) \\
 \cong \Hom_{\hat{R}_\fq}
 (R/P \otimes_{R}  \hat{R}_{\fq},(H_{P}
 ^{d}(R) \otimes_{R} R_P) \otimes_{R}   
 \hat{R}_\fq) \cong \\\Hom_{\hat{R}_\fq}
 (R/P \otimes_{R}  (\hat{R} \otimes_{\hat{R}}\hat{R}_{\fq}),(H_{P}
 ^{d}(R) \otimes_{R} R_P) \otimes_{R}   
 (\hat{R} \otimes_{\hat{R}}\hat{R}_{\fq})) \cong \Hom_{\hat{R}_\fq}(\frac{\hat{R}_\fq} { P \hat{R}_\fq},( H_{P\hat{R}}^{d}(\hat{R})_P)_\fq).
 \end{array}\]
 Therefore, by virtue of Lemma \ref{16}, $ \Hom_{R_{\fq \cap R}}(\frac{R_{\fq \cap R}}{PR_{\fq \cap R}} ,E_{R}(R/P)_{\fq 
 \cap R}) \cong \Hom_{R_{\fq \cap R}}(\frac{R_{\fq \cap R}}{PR_{\fq \cap R}} , (H_{P}^{d}(R)_P)_{\fq 
 \cap R})$ is a non-zero finitely generated  $R_{\fq \cap R}$-module. So, by Lemma \ref{15}, $\frac{R_{\fq \cap R}}{PR_{\fq \cap R}}$ is an Artinian ring.  Again, it is a contradiction because  $\frac{R_{\fq \cap R}}{PR_{\fq \cap R}}$ is a domain of dimension greater than one.

\end{proof}
Next we want to establish analogous result such Lemma \ref{5} for characteristic $p >0$. To show this we need some lemmas.
\begin{lemma}\label{7}
 Let $R$ be a regular  local ring which contains a field and $\I$ be an ideal of $R$. Let $\injdim_R(H_{\I}^i(R)) = \Dim_R(H_{\I}^i(R)) = c$. If $\mu^c(\fp,H_{\I}^i(R)) \neq 0$ for $\fp \in \Spec(R)$ then $\fp$ is a maximal ideal of $R$.
 \end{lemma} 
 \begin{proof}
 Let $\Dim(R) = n$.
We suppose on the contrary $\Ht_R \fp \leq n-1 $. Thus $\Dim_{R_\fp}(H_{\I}^i(R))_\fp \leq c-1 $ . Since $\mu^c(\fp,H_{\I}^i(R))\neq 0$, we deduce that $\injdim _{R_\fp}( H_{\I}^i(R)_\fp) =c$. But this is impossible because  in view of \cite[Theorem 3.4(b)]{L} and \cite[Theorem 1.4]{L1}, we must have $ \injdim _{R_\fp}( H_{\I}^i(R)_\fp)  \leq dim_{R_\fp}(H_{\I}^i(R))_\fp$.  
 \end{proof}
 \begin{lemma}\label{8}
 Let $(R,\fm)$ be a  local ring of dimension $n$. Let $\hat{R}$ denote the completion of $R$ with respect to the maximal ideal $\fm$. Let $M$ be an $R$-module. Then $\Dim_R(M) = \Dim_{\hat{R}}(M \otimes_R \hat{R})$. 
 \end{lemma}
 \begin{proof} 
   Let $ \Dim_R(M) =d$. There exists $\fp \in \Supp_R(M)$ such that  $  d = \Dim R/\fp = \Dim \hat{R}/\fp\hat{R}$. Thus there exists $\fq \in \Spec(\hat{R})$ such that $\fq$ is minimal over $\fp\hat{R}$ and $\Dim \hat{R}/\fq\hat{R} = d$. We show that $\fq \in \Supp_{\hat{R}} (M \otimes_R \hat{R})$ and so $ \Dim_{\hat{R}}(M \otimes_R \hat{R}) \geq d$. It is clear that $\fq \cap R = \fp$. Hence the natural map $R_\fp \to \hat{R}_{\fq}$ is faithfully flat. Thus $$   (  M \otimes_{R} \hat{R}) \otimes_{\hat{R}} \hat{R}_{\fq} \cong M \otimes_{R} \hat{R}_{\fq} \cong M \otimes _{R} (R_\fp \otimes_{R_\fp} \hat{R}_{\fq}) \cong  (  M \otimes_{R} R_\fp) \otimes_{R_\fp} \hat{R}_{\fq} .$$  So $(  M \otimes_{R} \hat{R})_{\fq} \neq 0$ as desired.  
   
    On the other hand let $ \Dim_{\hat{R}}(M \otimes_R \hat{R})  = c$. Thus there exists $\fq \in \Supp_{\hat{R}}(M \otimes_R \hat{R})$ such that $\Dim \hat{R}/\fq\hat{R} = c$.
  Let $\fq \cap R = \fp$. Thus $\Dim R/\fp = \Dim \hat{R}/\fp\hat{R} \geq \Dim \hat{R}/\fq = c$  . So we only need to show that $ \fp \in \Supp(M)$. It is obvious by the isomorphism $(M\otimes_{R} \hat{R})_{\fq} \cong ( M \otimes_{R} R_\fp) \otimes_{R_\fp} \hat{R}_{\fq}$.  
 \end{proof}  
 \begin{proposition}\label{9}
 Let $(R,\fm)$ be  a regular local ring of dimension $n$ containing a field and $\I$ be an ideal of $R$ such that $\Ht_R(\I) = d$.  Then $\injdim_R(H_{\I}^d(R)) = \Dim_R(H_{\I}^d(R)) $.
 \end{proposition}
 \begin{proof}
 Assume $\Ht_R(\I) =d$. Let $\min_R(I) = \{\fp_1,\ldots,\fp_s\} \cup \{\fq_1,\ldots,\fq_t\}$ such that $\Ht_R(\fp_i) =d $ and $\Ht_R(\fq_i) >d$. Set $\I^{\prime} := \fp_1 \cap \ldots \cap \fp_s$ and $\I^{\prime\prime} = \fq_1 \cap \ldots \cap \fq_t$.  We have the Mayer-vietoris sequence $$ H_{\I^{\prime} + \I^{\prime\prime}}^d(R) \to H_{\I^{\prime}}^d(R) \oplus H_{ \I^{\prime\prime}}^d(R) \to H_{\I }^d(R) \to H_{\I^{\prime} + \I^{\prime\prime}}^{d+1}(R).$$
 Since $H_{\I^{\prime} + \I^{\prime\prime}}^d(R) = H_{\I^{\prime} + \I^{\prime\prime}}^{d+1}(R) = H_{ \I^{\prime\prime}}^d(R) =0 $ we deduce that  $H_{\I}^d(R) \cong H_{I^{\prime} }^d(R)$. Thus without loss of generality,  we can assume that all minimal prime ideals of $\I$ have height $d$. 
 
 There   exists  the spectral sequence $H_{\fm}^i(H_{\I}^j(R)) \Rightarrow H_{\fm}^{i+j}(R)$.    By using Hartshorne-Lichtenbaum theorem, we easily see that  $ \injdim_R (H_{\I}^i(R)) \leq \Dim_R (H_{\I}^i(R)) \leq n- (i+1)$ for all $i > d$. So on the line $y+x = n $ of the spectral sequence $H_{\fm}^i(H_{\I}^j(R)) \Rightarrow H_{\fm}^{i+j}(R)$,  we have $(H_{\fm}^{n-i}(H_{\I}^i(R)) = 0$ for all $i >d $.  By the definition of the spectral sequence $H_{\fm}^i(H_{\I}^j(R)) \Rightarrow H_{\fm}^{i+j}(R)$ there exists a filtration $$0 \subseteq \ldots \subseteq	 F^tH_{n} \subseteq   F^{t-1}H_{n} \subseteq \ldots \subseteq F^sH_{n} = H_{\fm}^{n}(R)$$ of $H_{\fm}^{n}(R)$ such that $E_{\infty}^{i,n-i} \cong \frac{F^i H_{n}}{F^{i+1}H_{n}}$. Since $E_{\infty}^{n-d-i,d+i} = 0$ for all $i\geq 1$ then $E_{\infty}^{n-d,d} \cong H_{\fm}^n(R)$. Note that $E_{\infty}^{n-d,d} $ is the quotient of $ H_{\fm}^{n-d}(H_{\I}^d(R))$. Then $ H_{\fm}^{n-d}(H_{\I}^d(R))$ must be non-zero. It implies that 
 $ \Dim_R(H_{\I}^d(R)) =n-d \leq  \injdim_R(H_{\I}^d(R)) $.
\end{proof}
\begin{lemma}\label{10}
 Let $(R, \fm)$ be a regular  local ring of dimension $n$ which contains a field of characteristic $p>0$. Let $\fp$ be a prime ideal of $R$ such that $\Ht_R\fp =d < n-1$. Then $E_R(R/\fp) \cong H_\fp^d(R)_\fp$ with natural $\mathscr{F}$-module structure is not $\mathscr{F}$-finite.
\end{lemma}
\begin{proof}
Note that $E_R(R/\fp) \cong H_\fp^d(R)_\fp$ and by Remark \ref{3}(e), $E_R(R/\fp)$ has a natural $\mathscr{F}$-module structure.
 
First assume that $\Ht_R(\fp) = n-2$. By virtue of Proposition \ref{9}, $\injdim_R H_{\fp}^{n-2}(R) = 2$. Consider the following minimal injective resolution of $ H_{\fp}^{n-2}(R)$.
$$0 \to H_{\fp}^{n-2}(R) \to E_R(R/\fp) \to E^1 \to E^2 \to 0.$$
By Remark \ref{3} (c), this is a complex of $\mathscr{F}$-modules and $\mathscr{F}$-homomorphisms. In view of Lemma \ref{7} and Remark \ref{4} (d), $E^2 \cong E_R(R/\fm)^s$ where $s$ is a positive integer. Suppose on the contrary $E_R(R/\fp)$ is $\mathscr{F}$-finite. Then following Remark \ref{4} (a), $E^1$ must be $\mathscr{F}$-finite. There exist infinitely many primes $\fq \in \Spec(R)$ which $ \fp \subset \fq$ and $\Ht_R(\fq) = n-1$. For all such $\fq \in \Spec(R)$, in view of Proposition \ref{9}, $\injdim_{R_\fq} H_{{\fp}R_{\fq}}^{n-2}(R_{\fq}) = 1$ and considering Lemma \ref{7} we have $\mu^1(\fq , H_{\fp}^{n-2}(R)) >0$.   So we reach to a contradiction in view of Remark \ref{4} (d) , (e).

For the convenience of the reader, we bring a different proof of the fact that $\mu^1(\fq , H_{\fp}^{n-2}(R)) >0$ suggested by the referee. Suppose $\fq \supseteq \fp$ such that $\Ht_R(\fq) = n-1$. Claim $E_{\fq}^1 \neq 0$.

Suppose if possible $ E_{\fq}^1 = 0$. We have $H_{{\fp}R_{\fq}}^{n-2}(R_{\fq})$ is an injective $  R_{\fq}$-module. Choose $g$ such that $(\fp R_{\fq},g)$ is $\fq R_{\fq}$-primary. By using the standard long-exact sequence of local cohomology modules and Hartshorne-Lichtenbaum theorem, we have an exact sequence 
$$ 0 \to H_{{\fp}R_{\fq}}^{n-2}(R_{\fq}) \to (H_{{\fp}R_{\fq}}^{n-2}(R_{\fq}))_g \to H_{{\fq}R_{\fq}}^{n-1}(R_{\fq}) \to 0.$$
As $H_{{\fp}R_{\fq}}^{n-2}(R_{\fq})$ is an injective $  R_{\fq}$-module we get that $\fq  R_{\fq} \in \Ass _{R_{\fq}}(H_{{\fp}R_{\fq}}^{n-2}(R_{\fq}))_g$ which is a contradiction.

 Now suppose $\Ht_R(\fp) =  n-3 $. Let $\fq \in \Spec(R)$ such that $\Ht_R(\fq) = n-1$ and $\fp \subset \fq$. Suppose on the contrary that $E_R(R/\fp)$ is $\mathscr{F}$-finite. Thus $E_R(R/\fp)_{\fq}$ is $\mathscr{F}_{R_\fq}$-finite by Remark \ref{4} (f). This contradicts with the first step of the proof.
 
 By applying this argument for a finite step we prove the lemma.   
\end{proof}
\begin{remark}\label{11}
\begin{itemize}
\item[(i)] 
Adopt the above  notations of Lemma \ref{10}. Let $\fp$ be a prime ideal of $R$ such that $\Ht(\fp) \geq n-1$. Then it is easy to see that $E_R(R/\fp)$ is $\mathscr{F}$-finite. Indeed if $\fp = \fm$ then $E_R(R/\fm) = H_{\fm}^n(R)$. Otherwise let  $$0 \to H_{\fp}^{n-1}(R) \to E_R(R/\fp) \to E^1 \to  0$$ be the minimal injective resolution of $H_{\fp}^{n-1}(R)$. In view of Lemma \ref{7} and Remark \ref{4} (d), $E^1 \cong E_R(R/\fm)^s$ where $s$ is a positive integer . Thus by Remark \ref{4}$(a)$, $ E_R(R/\fp)$ is $\mathscr{F}$-finite.
\item[(ii)] Let $R = k[[x_1,\ldots,x_n]]$ and characteristic of $k$ is $0$. Let $\fp$ be a prime ideal of $R$ such that $\Ht(\fp) \geq n-1$. As $(i)$ one can easily see that $E_R(R/\fp)$ is holonomic.
\end{itemize}   
\end{remark}

Let $M$ be a finitely generated module over a Cohen-Macaulay ring $R$ such that $\injdim_{R}(M)$ is finite and therefore it equals to $\Dim R$. Then it is elementary to prove that if $\mu^{\Dim R}(\fp,M) >0$ then $\fp$ is a maximal ideal in $R$, use \cite[Proposition 3.1.13]{BH}. Although this fact is not true for $R$-module $M$ that is not finitely generated. For example Let $\fp$ be a prime ideal of $R$ and $M$ be the injective envelope of $R/\fp$.

For polynomial ring $R = k[x_1,\ldots,x_n]$ with field $k$ of characteristic zero, Puthenpurakal proved if $\injdim_R(H_{\I}^{i}(R)) = c$ and $\mu^c(\fp,H_{\I}^{i}(R)) >0$ for prime ideal $\fp$ of $R$, then $\fp$ is a maximal ideal of $R$, see \cite[Theorem 1.1]{puthen}. In the following proposition we generalize his theorem to the case that  $R$ is a  regular  local ring which contains a field. 
\begin{proposition}\label{12}
Let $R$ be a regular local ring of dimension $n$ which contains a field $k$.  Let $M$ be an $R$-module  such that  $\injdim_R(M) = c$ and  $\mu^c(\fp,M) \neq 0$ for  a prime ideal $\fp$ of $R$ . Assume that one of the following holds:
\begin{itemize}
\item[(i)] $k$ is a field of characteristic $p>0$ and $M$ be a $\mathscr{F}$-finite.
\item[(ii)] $R = k[[x_1,\ldots,x_n]]$ and characteristic of $k$ is $0$ and $M$ is a holonomic module.
\item[(iii)] $k$ is a field of characteristic $0$ and $M = H_I^j(R)_f$ where $I$ is an ideal of $R$ and $f \in R$.
\end{itemize}
 Then 
   $\Ht_R(\fp) \geq n-1$. 
\end{proposition}
\begin{proof}
We first show that $H_{\fp}^i(M)_{\fp}$ is an  injective $R$-module for all positive integer $i$. In case $(i)$, $H_{\fp}^i(M)_{\fp}$ is zero or  an $\mathscr{F}_{R_\fp}$-finite module of dimension $0$, see \ref{4}(c),(f). Then by  \ref{3}(d) and \ref{4}(d) $H_{\fp}^i(M)_{\fp} \cong E_R(R/\fp)^s$ where $s$ is a positive integer. In case $(ii)$, 
we note that $H_{\fp}^i(M)$ is a holonomic $\mathscr{D}$-module, see Remark \ref{2}(a). Let ${R_\fp}\hat{}$ denote the completion of $R_\fp$ with respect to the maximal ideal $\fp R_\fp$. It follows that $H_{\fp}^i(M)_{\fp}$ has a natural structure of $\mathscr{D}({R_\fp}\hat{}, k^\prime)$-module where $k^\prime$ is a suitable coefficient field of $ {R_\fp}\hat{}$,  see the proof of \cite[Theorem 2.4(b)]{L}. So, by Remark \ref{1}(a),  $H_{\fp}^i(M)_{\fp}$ is a direct sum of  copies of $E_{{R_\fp}\hat{}}({R_\fp}\hat{}/\fp{R_\fp}\hat{})$. But as an $R$-module $E_{{R_\fp}\hat{}}({R_\fp}\hat{}/\fp{R_\fp}\hat{})$ is isomorphic to $E_R(R/\fp)$, so $H_{\fp}^i(M)_{\fp}$ is an injective $R$-module. Also $H_{\fp}^i(M)_{\fp}$ is a direct sum of  finite copies of $E_R(R/\fp)$, see Remark \ref{2}(f). In case $(iii)$,  $(H_{\fp}^i(H_{\I}^j(R)))_\fp \cong E_R(R/\fp)^s$ where $s$ is a positive integer, see \cite[Theorem 3.4(b),(d)]{L}. Then $$ H_\fp^i(M)_\fp = H_\fp^i(H_I^j(R)_f)_\fp \cong ((H_\fp^i(H_I^j(R))_f)_\fp \cong (H_{\fp}^i(H_{\I}^j(R)))_\fp \otimes_R R_f \cong E_R(R/\fp)^s \otimes_R R_f .$$ Hence   $(H_{\fp}^i(M))_\fp \cong E_R(R/\fp)^t$ where $t$ is a positive integer. So $(H_{\fp}^i(M))_\fp$ is an injective $R$-module. 

Therefore, in three cases, we have $ \mu^0(\fp,H_{\fp}^c(M)) = \mu^c(\fp,M) >0$, see \cite[Lemma 1.4]{L}. Note that by the above discussion $H_{\fp}^c(M)_{\fp} \cong E_R(R/\fp)^s$ where $s >0 $ is an  integer.

 Suppose on the contrary $\Ht_R(\fp) \leq n-2$. Note that $\Ass_R (H_{\fp}^c(M))$ is finite, see Remark \ref{4}(e), \ref{2}(f) and \cite[Theorem 3.4(c)]{L}. Let $\Ass_R (H_{\fp}^c(M)) = \{ \fp,\fq_1,\ldots,\fq_m\}$. Look at the exact sequence:
$$0 \to \Gamma_{\fq_1\ldots\fq_m}(H_{\fp}^c(M)) \to H_{\fp}^c(M) \to H_{\fp}^c(M)/\Gamma_{\fq_1\ldots\fq_m}(H_{\fp}^c(M)) \to 0.$$
Since $\fp \subsetneqq \fq_i$, we have $\fp \notin \Ass_R \Gamma_{\fq_1\ldots\fq_m}(H_{\fp}^c(M))$. Keep in mind that  $$\Ass_R H_{\fp}^c(M) = \Ass_R \Gamma_{\fq_1\ldots\fq_m}(H_{\fp}^c(M)) \cup \Ass_R H_{\fp}^c(M)/\Gamma_{\fq_1\ldots\fq_m}(H_{\fp}^c(M)).$$ It follows  that $\Ass_R H_{\fp}^c(M)/\Gamma_{\fq_1\ldots\fq_m}(H_{\fp}^c(M)) = \{\fp\}$.

Let $g \in R\setminus \fp$. Then the following diagram commutes:
 
\[\xymatrix{
0 \ar[r]& \Gamma_{\fq_1\ldots\fq_m}(H_{\fp}^c(M))  \ar[d] \ar[r] & H_{\fp}^c(M) \ar[r] \ar[d] & H_{\fp}^c(M)/\Gamma_{\fq_1\ldots\fq_m}(H_{\fp}^c(M)) \ar[d]^{\eta} \ar[r] & 0\\
0 \ar[r] & (\Gamma_{\fq_1\ldots\fq_m}(H_{\fp}^c(M)))_g \ar[r] &H_{\fp}^c(M)_g \ar[r]& (H_{\fp}^c(M)/\Gamma_{\fq_1\ldots\fq_m}(H_{\fp}^c(M)))_g   \ar[r] &0.
}\]
Recall that $\injdim_R M =c$. Thus,
  there is an exact sequence $$ H_{(\fp,g)}^c(M) \to H_{\fp}^c(M) \to H_{\fp}^c(M)_g \to H_{(\fp,g)}^{c+1}(M) =0. $$ Hence, the natural map $ \eta$ is surjective. As $g\notin \fp$, we get that $\eta$ is also injective. Thus, $H_{\fp}^c(M)/\Gamma_{\fq_1\ldots\fq_m}(H_{\fp}^c(M)) = (H_{\fp}^c(M)/\Gamma_{\fq_1\ldots\fq_m}(H_{\fp}^c(M)))_g$ for all $g \in R\setminus \fp$. It follows that $ H_{\fp}^c(M)/\Gamma_{\fq_1\ldots\fq_m}(H_{\fp}^c(M))= (H_{\fp}^c(M)/\Gamma_{\fq_1\ldots\fq_m}(H_{\fp}^c(M)))_\fp$.
 
 Note that $(\Gamma_{\fq_1\ldots\fq_m}(H_{\fp}^c(M)))_\fp =0 $. We deduce that  $$  H_{\fp}^c(M)_\fp \cong  ( H_{\fp}^c(M)/\Gamma_{\fq_1\ldots\fq_m}(H_{\fp}^c(M)))_\fp \cong H_{\fp}^c(M)/\Gamma_{\fq_1\ldots\fq_m}(H_{\fp}^c(M)).$$

Now we prove the proposition
\begin{itemize}
 \item[(i)] 
  Clearly $H_{\fp}^c(M)/\Gamma_{\fq_1\ldots\fq_m}(H_{\fp}^c(M))$ is $\mathscr{F}$-finite. Putting this along with $(H_{\fp}^c(M))_\fp \cong E_R(R/\fp)^s$, we conclude that $E_R(R/\fp)$ is $\mathscr{F}$-finite. So we reach to a contradiction because  by Lemma \ref{10} $E_R(R/\fp)$ can not be $\mathscr{F}$-finite.
  \item[(ii)]Exactly same $(i)$: $H_{\fp}^c(M)/\Gamma_{\fq_1\ldots\fq_m}(H_{\fp}^c(M))$ is holonomic and it is contradicts with Lemma \ref{5}.

\item[(iii)] Let $\hat{R}$ be the completion of $R$ with respect to maximal ideal $\fm$. Then   $$(H_{\fp\hat{R}}^c(H_{\I\hat{R}}^j(\hat{R})_f))/\Gamma_{(\fq_1\ldots\fq_m)\hat{R}} (H_{\fp\hat{R}}^c(H_{\I\hat{R}}^j(\hat{R})_f)) \cong E_R(R/\fp)^s \otimes_R \hat{R}.$$ But $E_R(R/\fp)^s \otimes_R \hat{R}$ is not holonomic by Lemma \ref{6} \end{itemize} 
\end{proof}
\begin{example}\label{13}
 Let $R = k[[x,y,z]]$ be a power series ring over a field $k$ and let $I$ be the ideal $(xy, xz)R$ of $R$. Then  $\Dim_R H_{\I}^i(R) = 1$ and $\injdim_R ( H_{I}^i(R)) =0$, see  \cite[Examples 2.9]{hel}. Thus there exists $\fp \in \Ass_R(H_{\I}^i(R))$ such that $\Ht_R(\fp) = 2 $. It is well-known that   for all $R$-module $M$, $\fq \in \Ass_R(M)$ if and only if $\mu^0(\fq ,M)>0$. It follows that  $\mu^0(\fp ,H_{\I}^i(R))>0$. Thus the lower bound for the prime ideal $\fp$ in the Proposition \ref{12} is not strict.
\end{example}

 \section{Main theorem}
 In this section we prove our main  result about injective dimension of local cohomology. 
 \begin{theorem}\label{14}
 Let $(R,\fm)$ be a regular local ring which contains a field . Let $\I$ be an ideal of $R$.  Suppose that one of the following two conditions $(i)$ or $(ii)$ holds:
 \begin{itemize}
 \item[(i)]   $R$ is of prime characteristic  $p >0$ and $M$ is an $\mathscr{F}$-finite module. 
 \item[(ii)] $R$ is of characteristic  $0$ and $M = H_{\I}^i(R)_{f} $ for some $f \in R$. 
 \end{itemize}
 Then $\dim_R M-1 \leq\injdim_R M$.
 \end{theorem}
 \begin{proof}
 We prove the theorem by induction on $\dim (M)$. If $\dim (M) \leq 1$, we have  nothing to prove.   In case $(i)$, assume that for every $\mathscr{F}$-finite module  of the dimension less  than $n$ the theorem is true. In case $(ii)$, assume that for every $R$ module   $\mathcal{N} = H_I^j(R)_g$   of dimension less  than $n$ the theorem is true such that $ g \in R$.
 
 Now suppose $M$ be an $R$-module of dimension $n >1$ which satisfies either $(i)$ or $(ii)$.

Let $\fp$ be a prime ideal  of $R$ such that $\dim_{R_\fp}(M)_\fp = n-1$. Then $M_\fp$ satisfies induction hypothesis. Hence $n-2 \leq \injdim_{R_\fp} (M)_\fp$. If $ \injdim_{R_\fp} (M)_\fp = n-1 $ we are done. So we assume $\injdim_{R_\fp} (M)_\fp = n-2$. We claim that there is a prime ideal $\fq \subsetneq \fp$ such that $\mu^{n-2}(\fq , M) \neq 0$. Suppose on the contrary there is not such prime ideal. Pick $g \in \fp$ such that $\dim_{R_\fp} (\mathcal(M)_g) = n-1$. Then $\mathcal(M)_g$ satisfies the induction hypothesis , see Remark \ref{4}$(b)$. But $\injdim_{R_\fp} \mathcal(M)_g < n-2$ and this  contradicts with the induction hypothesis. 

So there is a prime ideal $\fq \subsetneq \fp$ such that $\mu^{n-2}(\fq , M) \neq 0$. In view of Proposition \ref{12}$(i),(iii)$ we conclude that $n-1 \leq \injdim M$, as desired.
 \end{proof}
\begin{remark}
Note that in view of Example \ref{13}, the lower bound in the main theorem is not strict.
\end{remark}

\subsection*{Acknowledgment}
I would like to thank the anonymous referee for his/her detailed review.



\begin{thebibliography}{99}
\bibitem{BBLSZ}
B. Bhatt, M. Blickle,G. Lyubeznik,A. Singh, W. Zhang, \emph{Local cohomology modules of a smooth $\mathbb{Z}$-algebra have finitely many associated primes},preprint(2014).

\bibitem{bjork}
J.-E. Bj$\ddot{o}$rk, \emph{Rings of Differential Operators},North Holland, Amsterdam,1979.
\bibitem{BH}
W. Bruns and J. Herzog,  \emph{Cohen-Macaulay rings}, Cambridge University Press, {\bf{39}}, Cambridge, (1998).
\bibitem{BS}
M. Brodmann and R.Y. Sharp, \emph{Local Cohomology: An Algebraic Introduction with Geometric Application}, Cambridge University Press, {\bf{60}}, Cambridge, (1998).
\bibitem{HS}
C. Huneke, R. Sharp, \emph{Bass numbers of local cohomology modules}, Trans.Amer.Math.Soc. {\bf 339}(1993),765-779.

\bibitem{hel}
M. Hellus, \emph{A note on the injective dimension of local cohomology modules}, Proceedeing of the American Mathematical Society.{\bf 136}(2008),2313-2321.
\bibitem{L}
G. Lyubeznik,  \emph{Finiteness properties of local cohomology modules(an application of D-modules to commutative algebra)}, Invent. Math.{\bf 113}(1993),41-55.
\bibitem{L1}
G. Lyubeznik,  \emph{F-modules : applications to local cohomolgy and $D$-modules in chracteristic $P > 0$}, J.Reine Angew. Math.{\bf 491}(1997),65-130.
\bibitem{puthen}
Tony J. Puthenpurakal, \emph{On injective resolution of local cohomology modules}, Illinois Journal of Mathematics. {\bf 58}(2014),709-718.
\bibitem{Mat}
H. Matsumura, \emph{Commutative Ring Theory}, Cambridge Studies in Advanced Math, \textbf{8}, (1986).
\end{thebibliography}
\end{document}